\documentclass{amsart}
\usepackage{url}
\newtheorem{thm}{Theorem}

\newtheorem{prop}[thm]{Proposition}
\theoremstyle{definition}

\theoremstyle{remark}

\newcommand{\R}{{\mathbb R}}
\newcommand{\Rd}{{{\mathbb R}^d}}

\begin{document}

\title{Uniqueness for the continuous wavelet transform}%
\author{H.-Q. Bui and R. S. Laugesen}%
\email{Huy-Qui.Bui@canterbury.ac.nz,Laugesen@illinois.edu}%


\date{\today}
\begin{abstract}
Injectivity of the continuous wavelet transform acting on a square integrable signal is proved under weak conditions on the Fourier transform of the wavelet, namely that it is nonzero somewhere in almost every direction. 

For a bounded signal (not necessarily square integrable), we show that if the continuous wavelet transform vanishes identically, then the signal must be  constant. 
\end{abstract}
\maketitle

\section{\bf Introduction}

Uniqueness for the Fourier transform acting on an integrable (or square integrable) function means that $\widehat{f}=0$ implies $f=0$. In other words, the Fourier transform is injective. For distributions, a related statement says that if the Fourier transform $\widehat{f}$ is supported at the origin then $f$ is a polynomial. In particular, if a bounded function has distributional Fourier transform supported at the origin, then it is a constant function. 

This note establishes analogous uniqueness results for the continuous wavelet transform. Given a function $\psi$ (which we call a \emph{wavelet}) and a function $f$ (which we call a \emph{signal}), the \emph{continuous wavelet transform} of $f$ with respect to $\psi$ is the function
\[
(W_\psi f)(s,t) = \langle f , \psi_{s,t} \rangle = \int_\Rd f(x) \overline{\psi_{s,t}(x)} \, dx , \qquad s>0 , \quad t \in \Rd ,
\]
where
\[
\psi_{s,t}(x) = \frac{1}{s^{d/2}} \psi \Big( \frac{x-t}{s} \Big) .
\]
Notice $s$ denotes the scale, and $t$ the translation. 

The scale and translation parameters vary continuously and so one calls $W_\psi$ the ``continuous'' wavelet transform, in distinction to the ``discrete'' wavelet transform which restricts to dyadic scales $s=2^j$ and translations $t=2^j k$. For more on wavelet transforms, readers may consult the texts of Daubechies \cite{D92}, 
Holschneider \cite{H95}, Mallat \cite{M09}, Meyer \cite{M92} and Pathak \cite{P09}. For precise relations between the continuous and discrete wavelet transforms, see Laugesen's work on translational averaging \cite{L01,L02}. 

We will present four uniqueness (or injectivity) results for the continuous wavelet transform. The first result deals with signals in $L^2$, under slightly weaker assumptions than the Calder\'{o}n admissibility condition. The second result handles signals in $L^p$ for $1 \leq p<\infty$. The third treats $L^\infty$, and the case of polynomially bounded signals. The fourth result considers tempered distributions. 

Our motivation comes from work of Sun and Sundararajan \cite{SS11} in theoretical economics. There the signal $f$ is a mixed partial derivative of the characteristic function of some \emph{attribution problem}.  Such problems arise in cooperative game theory as cost-sharing problems, in investment finance as performance analyses of investment portfolios, and in operations research settings in the analysis of production process performance. The wavelet $\psi$ represents the difference between two path-generated attribution methods.  With the help of our wavelet uniqueness results, these authors show, roughly speaking, that if two different path-generated attribution methods yield the same attributions, then the characteristic function must lie in some specific constrained class.

\section{\bf Wavelet uniqueness for square integrable signals}

We say that a function $\lambda$ on $\Rd$ is \emph{nontrivial in direction $\xi$} (where $\xi$ is a unit vector) if the set $\{ r>0 : \lambda(r\xi) \neq 0 \}$ has positive measure. For example, if $\lambda$ is continuous then nontriviality in a direction simply means $\lambda$ is nonzero at some point on the ray in that direction. 

Our first result treats uniqueness for the continuous wavelet transform in $L^2$.
\begin{prop} \label{integrable}
Assume $f, \psi \in L^2(\Rd)$. If 
\[
\langle f , \psi_{s,t} \rangle = 0 \qquad \text{for all $s>0 , \quad t \in \Rd$,}
\]
(in other words if $W_\psi f \equiv 0$), then 
\begin{equation} \label{underlying}
\int_0^\infty |\widehat{f}(r\xi)|^2 r^{d-1} \, dr \, \int_0^\infty |\widehat{\psi}(s\xi)|^2 s^{d-1} \, ds = 0
\end{equation}
for almost every unit vector $\xi$ in $\Rd$. 

In particular, if $W_\psi f \equiv 0$ then nontriviality of $\widehat{\psi}$ in almost every direction implies $f = 0$ a.e.
\end{prop}
In dimension $d=1$, the assumption that $\widehat{\psi}$ is nontrivial in almost every direction means that $\widehat{\psi}$ is nontrivial on each side of the origin: the two sets $ \{ r>0 : \widehat{\psi}(r) \neq 0 \}$ and $ \{ r>0 : \widehat{\psi}(-r) \neq 0 \}$ both have positive measure. 
This nontriviality assumption on the Fourier transform is the standard \emph{Tauberian condition} in harmonic analysis.

The symmetry of equation \eqref{underlying} reflects the interchangeability of the signal and the wavelet, which one sees by a simple change of variable:
\[
\langle f , \psi_{s,t} \rangle = \langle f_{1/s,-t/s} , \psi \rangle 
\]

\subsection*{Relation to the Calder\'{o}n condition.}  The Calder\'{o}n admissibility condition for continuous wavelets says that 
\[
\int_0^\infty |\widehat{\psi(s\xi)}|^2 \, \frac{ds}{s} = 1
\]
for almost every unit vector $\xi$ (see \cite[Section 2.4]{D92}, \cite{LWWW02}). This condition implies the hypothesis in Proposition~\ref{integrable} that $\widehat{\psi}$ is nontrivial in almost every direction, and indeed is stronger than that hypothesis because our $\widehat{\psi}$ need not vanish at the origin (or can vanish so slowly there that the Calder\'{o}n integral diverges). 

On the other hand, the Calder\'{o}n condition guarantees more than just uniqueness for the wavelet transform: it guarantees a Plancherel formula 
\[
\lVert f \rVert_2^2 = \int_0^\infty \int_\Rd |\langle f , \psi_{s,t} \rangle|^2 \, dt \, \frac{ds}{s^{d+1}} ,
\]
and hence (by polarization) a reproducing formula. Thus our Proposition assumes less and obtains less than the standard theory based on the Calder\'{o}n condition. 

\begin{proof}[Proof of Proposition~\ref{integrable}]
Define the Fourier transform with $2\pi$ in the exponent:
\[
\widehat{\psi}(\omega) = \int_\Rd \psi(x) e^{-2\pi i \omega \cdot x} \, dx .
\] 
Then by Parseval's identity and the vanishing of the wavelet transform we have
\[
0 
= \langle f , \psi_{s,t} \rangle \\
= \langle \widehat{f} , \widehat{\psi_{s,t}} \rangle .
\]
Direct calculation of the Fourier transform shows that $\widehat{\psi_{s,t}}(\omega) = \widehat{\psi}(s\omega) e^{-2\pi i \omega \cdot t} s^{d/2}$, and so the previous formula says that
\[
0 = \big[ \widehat{f}(\cdot) \overline{\widehat{\psi}(s \cdot)} \, \big] \widehat{\ }(-t) \qquad \text{for each $t \in \Rd$.}
\]
Since the integrable function $\widehat{f}(\cdot) \overline{\widehat{\psi}(s \cdot)}$ has vanishing Fourier transform, it must equal zero a.e., which means
\[
0 = \widehat{f}(\omega) \overline{\widehat{\psi}(s\omega)} \qquad \text{for almost every $\omega \in \Rd$,}
\]
for each $s > 0$. 

Next we square and multiply by $|\omega|^d s^{d-1}$, and then integrate to show that
\begin{align*}
0 
& = \int_\Rd |\widehat{f}(\omega)|^2 \int_0^\infty |\widehat{\psi}(s\omega)|^2 |\omega|^d s^{d-1} \, ds d\omega \\
& = \int_{S^{d-1}} \int_0^\infty |\widehat{f}(r\xi)|^2 r^{d-1} \, dr \, \int_0^\infty |\widehat{\psi}(s\xi)|^2 s^{d-1} \, ds \, dS(\xi)
\end{align*}
by expressing $\omega=r\xi$ in spherical coordinates and then changing variable with $s \mapsto s/r$. The integrand therefore vanishes for almost every $\xi$, which proves equation \eqref{underlying}. 

Finally, if $\widehat{\psi}$ is nontrivial in almost every direction then $ \int_0^\infty |\widehat{\psi}(s\xi)|^2 s^{d-1} \, ds$ is positive for almost every $\xi$, and so (by the preceding formula) the integral $\int_0^\infty |\widehat{f}(r\xi)|^2 r^{d-1} \, dr$ must vanish for almost every $\xi$. Thus $\widehat{f}=0$ a.e.\ and so $f=0$ a.e. The argument works also with the roles of $\psi$ and $f$ interchanged.
\end{proof}

\section{\bf Uniqueness for $p$-integrable signals}

Next we treat signals in $L^p$. The Fourier transform of the signal is a distribution, when $p>2$. We show that distribution has support at the origin, which implies the signal must vanish.
\begin{thm}\label{p-int}
Let $1 \leq p < \infty$ and $\frac{1}{p}+\frac{1}{p\prime}=1$. Assume $f \in L^p(\Rd), \psi \in L^{p\prime}(\Rd), \widehat{\psi} \in C^\infty(\Rd \setminus \{ 0 \})$, and that the wavelet transform vanishes identically:
\[
\langle f , \psi_{s,t} \rangle = 0 \qquad \text{for all $s>0 , \quad t \in \Rd$.}
\]
If $\widehat{\psi}$ is nontrivial in every direction then $f = 0$ a.e.
\end{thm}
The smoothness hypothesis on $\widehat{\psi}$ away from the origin simply means that some function $\nu \in C^\infty(\Rd \setminus \{ 0 \})$ represents the distribution $\widehat{\psi}$ when acting on test functions $\eta \in C^\infty_c(\Rd \setminus \{ 0 \})$; in other words $\widehat{\psi}[\eta] = \int_\Rd \nu(\omega) \eta(\omega) \, d\omega$. We will write $\widehat{\psi}$ to mean both the distributional Fourier transform and the function $\nu$, as there will be no danger of confusion.

This smoothness hypothesis on $\widehat{\psi}$ is satisfied if $\psi$ has compact support, because then $\widehat{\psi}$ is smooth on all of $\Rd$, including at the origin. 

For a non-compactly supported example, $\psi$ could be a Gaussian or one of its derivatives (such as the Mexican hat, the negative second derivative of the Gaussian), in which case the Fourier transform is smooth on all of $\Rd$. 

For an example where $\widehat{\psi}$ is not smooth at the origin, suppose $\psi(x)=1/\pi(1+x^2)$, which is the Poisson kernel in $1$ dimension. Then $\widehat{\psi}(\xi)=e^{-2\pi|\xi|}$ is smooth away from the origin, but not at the origin. Similarly if $\psi$ is the first derivative of the Poisson kernel (in which case $\psi$ has integral equal to zero) then $\widehat{\psi}(\xi)=2\pi\xi e^{-2\pi|\xi|}$, which is again smooth except at the origin. 

\begin{proof}[Proof of Theorem~\ref{p-int}]
We will show that the tempered distribution $\widehat{f}$ is supported at the origin. Then $f$ is a polynomial (see \cite[Corollary~2.4.2]{G08}), after suitable redefinition on a set of measure zero. Since $f$ belongs to $L^p$ by hypothesis, we conclude that the polynomial must be identically zero, as claimed in the theorem.

To show $\widehat{f}$ is supported at the origin, we start with a Schwartz function $\eta$ supported in $\Rd \setminus \{ 0 \}$. We must show $\widehat{f}[\eta]=0$, that is, $f[\widehat{\eta}]=0$.

\medskip
Write $\phi = \overline{\psi}$, so that (by hypothesis) $\widehat{\phi}$ is nontrivial in every direction.  The proof proceeds in a number of steps.

\medskip
Step 1. [Cut-off function.] For each unit vector $\xi^\prime$, choose $r>0$ such that $\widehat{\phi}(r\xi^\prime) \neq 0$.  By continuity of $\widehat{\phi}$, there exists a neighborhood $\Xi$ of $\xi^\prime$ on the unit sphere and a number $s>1$ such that $\widehat{\phi}(q\xi) \neq 0$ for all $\xi \in \Xi$ and all $q \in [r,sr]$. Cover the sphere with finitely many such neighborhoods $\Xi_1,\dots,\Xi_n$ having corresponding values $r_1,\dots,r_n$ and $s_1,\dots,s_n$. 

Choose a nonnegative function $\lambda \in C^\infty_c(\Rd \setminus \{ 0 \})$ such that $\lambda(q\xi)>0$ whenever $q \in [r_k,s_k r_k], \xi \in \Xi_k, k=1,\dots,n$; for example, one could take $\lambda$ to be a radially symmetric ``annular bump'' function that is zero near the origin and positive from radius $\min_k r_k$ out to radius $\max_k s_k r_k$.

\medskip
Step 2. [Satisfying the Calder\'{o}n condition.] Take $s = \min(s_1,\dots,s_n)$, and define a Schwartz function $\mu$ by letting its Fourier transform be
\[
\widehat{\mu}(\omega) = \frac{ \overline{\widehat{\phi}(\omega)} \lambda(\omega)}{ \sum_{j \in \mathbb{Z}} |\widehat{\phi}(s^j \omega)|^2 \lambda(s^j \omega)} , \qquad \omega \in \Rd \setminus \{ 0 \} .
\]
Note the term with $j=0$ in the denominator is positive at every point $\omega= q\xi$ with $q \in [r_k,s r_k], \xi \in \Xi_k, k=1,\dots,n$, and so by summing over $j$ we see the denominator is positive for every $\omega \neq 0$. Further, the series in the denominator converges because it involves only finitely many $j$ values (for each $\omega$), due to the compact support of $\lambda$ in $\Rd \setminus \{ 0 \}$. 

Hence 
\[
\sum_{j \in \mathbb{Z}} \widehat{\phi}(s^j \omega) \widehat{\mu}(s^j \omega) = 1 , \qquad \omega \in \Rd \setminus \{ 0 \} .
\]

\medskip
Step 3. [Calder\'{o}n reproducing formula.] Multiplying the result of Step 2 by $\eta(-\omega)$ shows that 
\[
\sum_{j \in \mathbb{Z}} \widehat{\phi}(s^j \omega) \widehat{\mu}(s^j \omega) \eta(-\omega) = \eta(-\omega) .
\]
Only a finite range of $j$-values (independently of $\omega$) is needed in the sum, because both $\eta$ and $\widehat{\mu}$ have compact support in $\Rd \setminus \{ 0 \}$. 

Applying the inverse Fourier transform gives a convolution formula:
\[
\sum_{j \in \mathbb{Z}} \phi_{s^j} * \mu_{s^j} * \widehat{\eta} = \widehat{\eta} ,
\]
where we use the notation $\phi_{s^j}(t)=\phi(t/s^j)/s^{jd}$ and so on. Convergence of the sum is guaranteed, because only finitely many $j$-values are summed. 

\medskip
Step 4. [Applying the reproducing formula.]  Hence the distribution $f$ acts on the Schwartz function $\widehat{\eta}$ according to
\begin{align*}
f[\widehat{\eta}]
& = \int_\Rd f(x) \widehat{\eta}(x) \, dx \\
& = \sum_{j \in \mathbb{Z}} \int_\Rd f(x) (\phi_{s^j} * \mu_{s^j} * \widehat{\eta})(x) \, dx \qquad \text{by Step 3} \\
& = \sum_{j \in \mathbb{Z}} \int_\Rd \int_\Rd f(x) \phi_{s^j}(x-t) \, (\mu_{s^j} * \widehat{\eta})(t) \, dt dx \\
& = \sum_{j \in \mathbb{Z}} \int_\Rd s^{-jd/2} \langle f , \psi_{s^j,t} \rangle \, (\mu_{s^j} * \widehat{\eta})(t) \, dt \\
& = 0 ,
\end{align*}
since $\langle f , \psi_{s^j,t} \rangle = 0$ by the hypothesis that the wavelet transform vanishes. 

Thus the distribution $\widehat{f}$ is supported at the origin, as we needed to show.
\end{proof}

\subsection*{Comment on the literature.} The construction of $\mu$ in Steps 1 and 2 originated in Calder\'{o}n's work on his reproducing formula \cite{C64}. The discrete form used above (involving sums rather than integrals, over the scales) is a special case of a result by Str\"{o}mberg and Torchinsky \cite[Chapter V, Lemma 6]{ST89} (and see also \cite{H74}).

Both the continuous and discrete Calder\'{o}n reproducing formulas play an important role in the characterization of classical function spaces in mathematical analysis \cite{BPT96,BPT97,FJW91,W10}.

\section{\bf Uniqueness for bounded or polynomially bounded signals}

Next we treat signals in $L^\infty$. Uniqueness of bounded signals will hold only up to additive constants. More generally, we handle signals that grow polynomially. 
\begin{thm}\label{bounded}
Assume $f$ is locally integrable with at most polynomial growth, meaning $f(x)(1+|x|)^{-k} \in L^\infty(\Rd)$ for some nonnegative integer $k$.  Assume $\psi$ is integrable with $\psi(x)(1+|x|)^k \in L^1(\Rd)$ and $\widehat{\psi} \in C^\infty(\Rd \setminus \{ 0 \})$, and suppose the wavelet transform vanishes identically:
\[
\langle f , \psi_{s,t} \rangle = 0 \qquad \text{for all $s>0 , \quad t \in \Rd$.}
\]
If $\widehat{\psi}$ is nontrivial in every direction, then $f$ is equal almost everywhere to a polynomial of degree $\leq k$. For example, if $f$ is bounded ($k=0$) then $f$ must be constant.
\end{thm}

\begin{proof}[Proof of Theorem~\ref{bounded}]
The proof proceeds exactly as for Theorem~\ref{p-int}, except that in the first paragraph of the proof, we do not know $f$ belongs to $L^p$ and so we cannot conclude the polynomial is identically zero. Instead, we simply conclude the polynomial has degree at most $k$. 
\end{proof}

\subsection*{Vanishing moments} The conclusion of the theorem forces the wavelet $\psi$ to have vanishing moments. For example, if $f$ is bounded (the case $k=0$) then $f$ is constant by the theorem; and so either $f$ is identically zero or else $\psi$ has integral zero, 
\[
\int_\Rd \psi(x) \, dx = 0,
\]
because of the hypothesis that $\langle f , \psi_{1,0} \rangle = 0$. 

For higher moments, let us consider the $1$ dimensional case and write the polynomial $f$ as $f(x)=\sum_{\ell=0}^m c_\ell x^\ell$ for some coefficients $c_\ell$, where $m=\deg(f)$. Suppose $f$ is not identically zero, so that the leading coefficient is nonzero, $c_m \neq 0$. Then $\psi$ has vanishing moments up to order $m$, meaning
\[
\int_\R x^\ell \psi(x) \, dx = 0 , \qquad \ell = 0,1,\dots,m,
\]
we now show. From $\langle f , \psi_{1,t} \rangle = 0$ we deduce that $\sum_{\ell=0}^m c_\ell \int_\R (x+t)^\ell \, \overline{\psi(x)} \, dx = 0$. Differentiating $m$ times with respect to $t$ and then setting $t=0$ shows that $\int_\R \psi(x) \, dx = 0$, since $c_m \neq 0$. Differentiating $m-1$ times with respect to $t$ and setting $t=0$ then shows that $\int_\R x \psi(x) \, dx = 0$. Repeating this argument down to the $0$th derivative establishes the claimed vanishing moments. 

\smallskip
We conclude by extending the uniqueness result to signals that are general tempered distributions.
\begin{thm}\label{distribution}
Assume $f$ is a tempered distribution, $\psi$ is a Schwartz function, and the wavelet transform vanishes identically:
\[
\langle f , \psi_{s,t} \rangle = 0 \qquad \text{for all $s>0 , \quad t \in \Rd$.}
\]
If $\widehat{\psi}$ is nontrivial in every direction, then $f$ is a polynomial. 
\end{thm}
The proof requires only a rephrasing into distributional language of Step 4 in the proof of Theorem~\ref{p-int}. We leave this task to the reader.

\section*{Acknowledgments} We thank Y. Sun and M. Sundararajan for asking us about the continuous wavelet uniqueness problem, and explaining its implications in economic problems.

Laugesen thanks the Department of Mathematics and Statistics at the University of Canterbury, New Zealand, for hosting him during this research.

\medskip

\end{document}